\documentclass{article}
\usepackage{graphicx} 

\usepackage[utf8]{inputenc}
\usepackage{amssymb}
\usepackage{amsthm}
\usepackage{amsmath}
\usepackage{graphicx,verbatim,color,booktabs,mathtools}
\usepackage{float}
\usepackage{caption}
\usepackage{subcaption}
\usepackage[hidelinks]{hyperref}

\newcommand{\norm}[1]{\lVert#1\rVert}
\newcommand{\R}[0]{\mathbb{R}}
\newcommand{\Z}[0]{\mathbb{Z}}

\newcommand\funding[1]{\protect\\ \hspace*{15.37pt}{\bfseries Funding:} #1}
\newcommand{\email}[1]{\protect\href{mailto:#1}{#1}}

\newtheorem{thm}{Theorem}
\newtheorem{lem}{Lemma}

\title{On the kissing number of the cross-polytope
}

\author{Niklas Miller\thanks{Department of Mathematics and Systems Analysis, Aalto University, Finland, E-mail: \email{niklas.miller@aalto.fi}.\funding{This work has been supported by the Research Council of Finland under Grant No. 351271 (PI~C. Hollanti).}}}

\date{\today}

\begin{document}

\maketitle

\begin{abstract}
     A new upper bound $\kappa_T(K_n)\leq 2.9162^{(1+o(1))n}$ for the translative kissing number of the $n$-dimensional cross-polytope $K_n$ is proved, improving on Hadwiger's bound $\kappa_T(K_n)\leq 3^n-1$ from 1957. Furthermore, it is shown that there exist kissing configurations satisfying $\kappa_T(K_n)\geq 1.1637^{(1-o(1))n}$, which improves on the previous best lower bound $ \kappa_T(K_n)\geq 1.1348^{(1-o(1))n}$ by Talata. It is also shown that the lattice
    kissing number satisfies $\kappa_L(K_n)< 12(2^n-1)$ for all $n\geq 1$, and that the lattice $D_4^+$ is the unique lattice, up to signed permutations of coordinates, attaining the maximum lattice kissing number $\kappa_L(K_4)=40$ in four dimensions.
\end{abstract}



\section{Introduction}

The translative \emph{kissing number problem} asks for the maximum number of non-overlapping translates of a given convex body that can touch a central one. The kissing number of the $n$-dimensional $\ell^2$-ball $B_n=\{x\in\mathbb{R}^n:\sum_{i=1}^n x_i^2\leq 1\}$ is known in dimensions $1$--$4$, $8$ and $24$ \cite{schutte1952problem,musin2008kissing,levenshtein1979bounds,odlyzko1979new}, and the asymptotic bounds $2^{0.2075{(1+o(1))n}}\leq \kappa_T(B_n)\leq 2^{0.401(1+o(1))n}$ follow from the works of  Wyner \cite{wyner1965capabilities}, Kabatiansky and Levenshtein \cite{kabatiansky1978bounds} and Levenshtein \cite{levenshtein1979bounds}. The kissing number problem for convex bodies different from $B_n$ are equally interesting and important. See \cite{Swanepoel2018} for a survey on kissing numbers of convex bodies and related problems.

In this article, we focus on the kissing number problem for the $n$-dimensional cross-polytope $K_n=\{x\in \mathbb{R}^n:\sum_{i=1}^n |x_i|\leq 1\}$, \emph{i.e.}, the convex hull of the $n$ $\pm$ pairs of standard basis vectors in $\mathbb{R}^n$. The cross-polytope is, together with the cube and the simplex, the only regular polytope in dimension $n\geq 5$. The exact value of $\kappa_T(K_n)$ is known only in dimensions $1$--$3$ where it is $2$, $8$ and $18$ \cite{larman1999kissing}, respectively. In each of these dimensions, the kissing configuration consists of the shortest non-zero vectors of a lattice -- the $D_n$ lattice, defined as the set of elements of $\mathbb{Z}^n$ whose coordinates sum up to an even number. In dimension four, the shortest vectors of the $D_4$ lattice gives a kissing configuration with $32$ vectors. However, the lattice $D_4$ together with the translate $D_4+(\frac{1}{2},\frac{1}{2},\frac{1}{2},\frac{1}{2})$ produce a kissing configuration with $40$ vectors, which we show is optimal in Theorem \ref{thm:cross-polytope-4}. Moreover, the optimal lattice kissing configuration is unique up to equivalence (see Section \ref{sec:preliminaries} for definitions).

Using a volume argument, Hadwiger \cite{hadwiger1957treffanzahlen} proved that the bound $\kappa_T(C)\leq 3^n-1$ holds for any $n$-dimensional convex body, a bound which remains the best general upper bound for convex bodies. In Section \ref{sec:asymptotics}, we use a variant of Blichfeldt's enlargement method \cite{Blichfeldt1929} to show that

\begin{thm}\label{thm:main0}
    $\kappa_T(K_n)\leq 2.9162^{(1+o(1))n}$ as $n\to\infty$.
\end{thm}

This seems to be the first improvement on Hadwiger's bound for $\ell^p$-balls different from $B_n$ (note that $\kappa_T(Q_n)=3^n-1$, where $Q_n$ is the $n$-cube). The lattice kissing number bound $\kappa_L(C)\leq 2(2^n-1)$ for strictly convex $C$ is a classical bound by Minkowski \cite{minkowski1961diophantische}, see also \cite[Theorem 30.2]{gruber2007convex}. While $K_n$ is not strictly convex, it is possible to obtain an equally strong inequality for $\kappa_L(K_n)$, up to a constant factor. See Section \ref{sec:lattice_kissing}.

\section{Preliminaries}
\label{sec:preliminaries}

In this section, we recall basic facts about convex bodies and kissing configurations. By a convex body $C$, we mean a compact and convex subset of $\mathbb{R}^n$ with non-empty interior. The set of centrally symmetric convex bodies in $\mathbb{R}^n$ is denoted $\mathcal{K}_n^0$. The symmetry group of $C\in\mathcal{K}_n^0$ is defined as the group $\mathcal{G}_C=\{T\in O_n(\mathbb{R}):T(C)=C\}$, where $O_n(\mathbb{R})$ is the orthogonal group consisting of distance-preserving linear maps $\mathbb{R}^n\to\mathbb{R}^n$. For instance, the symmetry group $S(K_n)$ of the unit cross-polytope $K_n=\{x\in\mathbb{R}^n: \sum_{i=1}^n |x_i|\leq 1\}$ is the group of order $2^n n!$ consisting of signed permutations of the coordinates. Any centrally symmetric convex body $C$ induces a norm on $\mathbb{R}^n$ via $\norm{x}_C=\inf \{t>0: x\in t C\}$ for $x\in\mathbb{R}^n$.

By a kissing configuration of a convex body $C$, we mean a set $S\subset\mathbb{R}^n$ such that the translates $x+\frac{1}{2}C,\ x\in S$ have disjoint interiors and each $x+\frac{1}{2}C$ intersects $\frac{1}{2}C$ only at its boundary. Note that $S$ is a kissing configuration of $C$ if and only if it is a kissing configuration of the central symmetrization $\frac{1}{2}(C-C)\in\mathcal{K}_n^0$, so without loss of generality we assume henceforth that $C\in\mathcal{K}_n^0$. Under this assumption, a kissing configuration is equivalently a set $S\subseteq \mathbb{R}^n$ such that $\norm{x}_C=1$ for all $x\in S$ and $\norm{x-x'}_C\geq 1$ for all $x,x'\in S$, $x\neq x'$.

The maximum size of a kissing configuration is denoted $\kappa_T(C)$ and sometimes called the Hadwiger number of $C$. The maximum size of a kissing configuration of $C$ consisting of the shortest non-zero vectors (with respect to the norm $\norm{\cdot}_C$) of a lattice is denoted $\kappa_L(C)$. Recall that a lattice $\Lambda$ is the $\mathbb{Z}$-span of a basis of $\mathbb{R}^n$. Clearly, $\kappa_L(C)\leq \kappa_T(C)$. The condition that the shortest non-zero vectors of a lattice $\Lambda$ form a kissing configuration is equivalent to the condition that $\min_{x\in\Lambda\setminus\{0\}}\norm{x}_C=1$, and the condition that the kissing configuration is as large as possible is equivalent to the set $\mathcal{M}_C(\Lambda)=\{x\in\Lambda:\norm{x}_C=1\}$ having maximal cardinality.

We call two lattices $\Lambda_1$ and $\Lambda_2$ $C$-equivalent, if $\Lambda_1=\sigma(\Lambda_2)$ for some $\sigma\in\mathcal{G}_C$. Note that $|\mathcal{M}_C(\Lambda)|$ depends only on the $C$-equivalence class of $\Lambda$. Define the support of a vector $x=(x_1,\dots,x_n)\in\R^n$ to be the set $\text{supp}(x)=\{i:x_i\neq 0\}$. We say that $i$ is a max-coordinate of $x$ if $|x_i|=\norm{x}_\infty$. We define the set $D_n^{+}\subset \mathbb{R}^n$ to be the union of $D_n$ and its translate $D_n+(\frac{1}{2},\dots,\frac{1}{2})$, where $D_n=\{x\in\mathbb{Z}^n: \sum_{i=1}^n x_i\equiv0\pmod{2}\}$. It is easily seen that $D_n^{+}$ is a lattice if and only if $n$ is even. For instance, $D_8^+$ is isometric to the $E_8$ lattice, while $D_4^+$ is a scaled Hadamard rotation of the integer lattice $\mathbb{Z}^4$. We define the ($\ell^1$-norm) covering radius of a discrete set $S\subset\mathbb{R}^n$ to be the quantity $\rho(S)=\sup_{y\in\mathbb{R}^n}\inf_{x\in S}\norm{x-y}_1$. We call the points $y\in \mathbb{R}^n$ achieving the supremum deep holes of $S$.

\section{The lattice kissing number of $K_n$}
\label{sec:lattice_kissing}

In this section, we focus on the lattice kissing number problem for $K_n$, and prove the following theorem.

\begin{thm}
    \label{thm:cross-polytope-4}
    $\kappa_L(K_4)=40$, achieved precisely by lattices $K_4$-equivalent to $D_4^+$.
\end{thm}
\begin{proof}
    Let $L\subset\R^4$ be a lattice and let $\mathcal{M}$ be a set containing exactly one of each $\pm$ pairs of vectors in $\mathcal{M}_{K_4}(L)$. Suppose that $|\mathcal{M}|\geq 20$. Note that if $x,y\in\mathcal{M}$ are distinct vectors such that $x\equiv y\pmod{2L}$, then $\frac{x\pm y}{2}\in \mathcal{M}_{K_4}(L)$ and the triangle inequality implies that $\text{supp}(x)\cap\text{supp}(y)=\emptyset$, and in particular $|\text{supp}(x)|\leq 2$ or $|\text{supp}(y)|\leq 2$.
    
    Since there are $2^4-1=15$ non-zero classes in $L/2L$, the following is true: $\mathcal{M}$ contains $k\in\{0,1,2,3,4\}$ vectors whose support has size one and $5-k$ pairs $\{x_1^i,x_2^i\}_{i=1}^{5-k}$ of vectors such that $x_1^i\equiv x_2^i\pmod{2L}$ and $|\text{supp}(x_1^i)|=|\text{supp}(x_2^i)|=2$ for all $i=1,\dots,5-k$. If $x,y\in\mathcal{M}$ are distinct vectors sharing a max-coordinate, then $\norm{x}_\infty\leq \frac{1}{2}$ or $\norm{y}_\infty \leq\frac{1}{2}$. In particular, if $|\text{supp}(x)|\geq 2$, then $|x_i|\leq \frac{1}{2}$ for all $i\in \cup_{x\in \mathcal{M}:|\text{supp}(x)|=1}\text{supp}(x)$. Then by the pigeonhole principle, among the $2(5-k)$ vectors $\{x_1^i,x_2^i\}_{i=1}^{5-k}$, there are at most $4-k$ with max-norm $>\frac{1}{2}$. In particular, there exists a pair $\{x_1^{i_1},x_2^{i_1}\}$ satisfying $\norm{x_1^{i_1}}_\infty=\norm{x_2^{i_1}}_\infty=\frac{1}{2}$, and moreover, $\mathcal{M}$ contains in addition to the pair $\{x_1^{i_1},x_2^{i_1}\}$ at least $4-k$ vectors $\{y_i\}_{i=1}^{4-k}$ with $\norm{y_i}_\infty=\frac{1}{2}$ and $|\text{supp}(y_i)|=2$ for all $i=1,\dots,4-k$. Recall that $\text{supp}(x_1^{i_1})\cap\text{supp}(x_2^{i_1})=\emptyset$. Thus, replacing $L$ by a $K_4$-equivalent lattice, we may suppose that $x_1^{i_1}=(\frac{1}{2},\frac{1}{2},0,0)$ and $x_2^{i_1}=(0,0,\frac{1}{2},\frac{1}{2})$. The vectors $x_1^{i_1}$, $x_2^{i_1}$, $\frac{x_1^{i_1}+x_2^{i_1}}{2}=(\frac{1}{4},\frac{1}{4},\frac{1}{4},\frac{1}{4})$ together with the vectors $\{y_i\}_{i=1}^{4-k}$ and $\{x\in\mathcal{M}:|\text{supp}(x)|=1\}$ span a lattice which contains the lattice $L'$ generated by $\mathbb{Z}^4$, $(\frac{1}{2},\frac{1}{2},0,0)$, $(0,0,\frac{1}{2},\frac{1}{2})$ and $(\frac{1}{4},\frac{1}{4},\frac{1}{4},\frac{1}{4})$.

    Since $|\mathcal{M}_{K_4}(L')|=20$, the containment $L'\subsetneq L$ is strict. Since $\rho(L')\leq 1$ by the next lemma, $L$ contains a deep hole of $L'$, which is also a deep hole of the lattice $\mathcal{H}_2\oplus\mathcal{H}_2$ generated by $\Z^4$, $(\frac{1}{2},\frac{1}{2},0,0)$ and $(0,0,\frac{1}{2},\frac{1}{2})$. Thus, by the next lemma, $L\setminus L'$ contains a vector of one of the following types.
    \begin{align*}
        v_1&=(x_1,\frac{1}{2}-x_1,x_2,\frac{1}{2}-x_2),\quad 0\leq x_1,x_2\leq \frac{1}{2}, \\
        v_2&=(x_1,x_1-\frac{1}{2},x_2,\frac{1}{2}-x_2),\quad 0\leq x_1,x_2\leq \frac{1}{2},\\
        v_2'&=(x_1,\frac{1}{2}-x_1,x_2,x_2-\frac{1}{2}),\quad 0\leq x_1,x_2\leq \frac{1}{2},\\
        v_3&=(x_1,x_1-\frac{1}{2},x_2,x_2-\frac{1}{2}),\quad 0\leq x_1,x_2\leq \frac{1}{2}.
    \end{align*}
Suppose $v_1\in L$. Then 
\begin{align*}
    \norm{(x_1,\frac{1}{2}-x_1,x_2,\frac{1}{2}-x_2)-(\frac{1}{4},\frac{1}{4},\frac{1}{4},\frac{1}{4})}_1\leq 1
\end{align*}
for all $0\leq x_1,x_2\leq \frac{1}{2}$, with equality if and only if $x_1,x_2\in\{0,\frac{1}{2}\}$. Hence, $L$ contains $(\frac{1}{2},0,\frac{1}{2},0)$. The lattice $\frac{1}{2}D_4^+$ is generated by $L'$ and $(\frac{1}{2},0,\frac{1}{2},0)$. We conclude that $\frac{1}{2}D_4^+\subseteq L$.

Suppose $v_2\in L$. Then 
\begin{align*}
    \norm{(x_1,x_1-\frac{1}{2},x_2,\frac{1}{2}-x_2)-(\frac{1}{4},\frac{1}{4},\frac{1}{4},\frac{1}{4})}_{1}\leq 1
\end{align*}
for all $\frac{1}{4}\leq x_1\leq\frac{1}{2}$, $0\leq x_2\leq \frac{1}{2}$, with equality if and only if $x_2\in\{0,\frac{1}{2}\}$. Similarly,
\begin{align*}
   \norm{(x_1,x_1-\frac{1}{2},x_2,\frac{1}{2}-x_2)-(\frac{1}{4},-\frac{3}{4},\frac{1}{4},\frac{1}{4})}_{1}\leq 1
\end{align*}
for all $0\leq x_1\leq\frac{1}{4}$, $0\leq x_2\leq \frac{1}{2}$, with equality if and only if $x_2\in\{0,\frac{1}{2}\}$. Thus, $v_2=(x_1,x_1-\frac{1}{2},0,\frac{1}{2})$
 or $v_2=(x_1,x_1-\frac{1}{2},\frac{1}{2},0)$. This implies that $L$ contains $(2x_1-\frac{1}{2},2x_1-\frac{1}{2},0,0)$. Since $-\frac{1}{2}\leq 2x_1-\frac{1}{2}\leq \frac{1}{2}$, we must have $x_1\in\{0,\frac{1}{4},\frac{1}{2}\}$. Thus, $L$ contains $(\frac{1}{2},0,\frac{1}{2},0)$ or $(\frac{1}{4},-\frac{1}{4},0,\frac{1}{2})$, so it contains $\frac{1}{2}D_4^+$ or the lattice $L_0$ spanned by $L'$ and $(\frac{1}{4},-\frac{1}{4},0,\frac{1}{2})$. The assumption $v_2'\in L$ leads to the same conclusion.

 Now suppose $v_3\in L$. Then $2v_3+(0,1,0,1)=(2x_1,2x_1,2x_2,2x_2)\in L$. It is easy to see that this implies $(2x_1,2x_1,2x_2,2x_2)\in L'$. Then the only possibilities are $x_1,x_2\in\{0,\frac{1}{4},\frac{1}{2}\}$ or $x_1,x_2\in\{\frac{1}{8},\frac{3}{8}\}$. In the former case, $L$ contains $\frac{1}{2}D_4^+$ or $L_0$. In the latter case, $L$ contains the lattice $L_1$ spanned by $L'$ and $(\frac{1}{8},-\frac{3}{8},\frac{1}{8},-\frac{3}{8})$ or the lattice $L_1'$ spanned by $L'$ and $(\frac{1}{8},-\frac{3}{8},\frac{3}{8},-\frac{1}{8})$. Note that $L_1$ is $K_4$-equivalent to $L_1'$.

In conclusion, $L$ contains $\frac{1}{2}D_4^+$, $L_0$, $L_1$ or $L_1'$. If the containment is strict, then $L$ contains an additional deep hole as above, which implies that $L$ contains two of these lattices. This however implies that $\min_{x\in L\setminus\{0\}}\norm{x}_1<1$, a contradiction. In conclusion, $L\in\{\frac{1}{2}D_4^+,L_0,L_1,L_1'\}$. We check that
\begin{align*}
    |\mathcal{M}_{K_4}(\frac{1}{2}D_4^+)|=40,\quad
    |\mathcal{M}_{K_4}(L_0)|=36,\quad
    |\mathcal{M}_{K_4}(L_1)|=|\mathcal{M}_{K_4}(L_1')|=28.
\end{align*}
We are left with only one possibility: $L=\frac{1}{2}D_4^+$.

\end{proof}

\begin{lem}
\label{lem:h2_covering}
Let $\mathcal{H}_2$ be the planar lattice generated by $(1,0)$ and $(\frac{1}{2},\frac{1}{2})$. Then the set of deep holes of the lattice $\mathcal{H}_2\oplus \mathcal{H}_2\subset \mathbb{R}^4$ is the set $$\left\{(x_1,\pm(\frac{1}{2}-x_1),x_2,\pm(\frac{1}{2}-x_2)):0\leq x_1,x_2\leq \frac{1}{2}\right\}+\mathcal{H}_2\oplus\mathcal{H}_2.$$
In particular, $\rho(\mathcal{H}_2\oplus\mathcal{H}_2)=1$.
\end{lem}
\begin{proof}
The deep holes of a direct sum $L_1\oplus L_2$ of lattices $L_1$, $L_2$ are given by $(x,y)$, where $x$ is a deep hole of $L_1$ and $y$ is a deep hole of $L_2$. Thus, it suffices to determine the deep holes of $\mathcal{H}_2$. Note that $\mathcal{H}_2$ is a scaled rotation of the lattice $\mathbb{Z}^2$, and $K_2$ is a scaled rotation, with the same factor, of the cube $Q_2=\{x\in\mathbb{R}^2:\norm{x}_\infty\leq 1\}$. The deep holes of $\mathbb{Z}^2$ with respect to the norm $\norm{\cdot}_\infty$ are the sets of vertical and horizontal lines passing through $(\frac{1}{2},\frac{1}{2})+\mathbb{Z}^2$. Now rotating back to the original situation, we see that the deep holes of $\mathcal{H}_2$ consist of the translates by vectors in $\mathcal{H}_2$ of the line segments $\{(x,\pm(\frac{1}{2}-x)):0\leq x\leq \frac{1}{2}\}$. 
\end{proof}

In a recent article \cite[Theorem 3.5]{miller2019kissing}, the authors prove that $\kappa_L(B_n^p)=O(\frac{n}{p}e^{\frac{n}{p}})$ for all $0<p\leq 2$, where $B_n^p=\{x\in\mathbb{R}^n:\sum_{i=1}^n |x_i|^p\leq 1\}$ is the unit $\ell^p$-ball. The next theorem improves this to $O(2^n)$ for the $\ell^1$-ball, which is similar to the upper bound $\kappa_L(C_s)\leq 2(2^{n}-1)$ which holds for strictly convex $C_s\in\mathcal{K}_n^0$.

\begin{thm}
\label{thm:lattice_kissing_bound}
    $\kappa_L(K_n)< 12(2^n-1)$ for all $n\geq 1$.
\end{thm}
\begin{proof}
    Let $L\subset\mathbb{R}^n$ be a lattice and $\mathcal{M}$ a set of representatives of $\pm$ pairs of vectors in $\mathcal{M}_{K_n}(L)$. Suppose that $|\mathcal{M}|\geq 6(2^n-1)$. As in the proof of Theorem \ref{thm:cross-polytope-4}, this implies that $\mathcal{M}$ contains $\ell\geq 2^n-1$ vectors whose support has size at most $\lfloor\frac{n}{6}\rfloor$. But the number of vectors whose support has size $k\leq n$ in any kissing configuration in $\mathbb{R}^n$ is at most $\binom{n}{k}3^k$ by Hadwiger's bound. Thus, $\ell\leq \sum_{k=1}^{\lfloor\frac{n}{6}\rfloor}\binom{n}{k}3^k\leq \lfloor \frac{n}{6}\rfloor\binom{n}{\lfloor\frac{n}{6}\rfloor} 3^{\lfloor\frac{n}{6}\rfloor}$, which contradicts $\ell\geq 2^n-1$ by routine estimation.
\end{proof}

\section{Asymptotic bounds for $\kappa_T(K_n)$}
\label{sec:asymptotics}

\subsection{An asymptotic upper bound for $\kappa_T(K_n)$}

In this section, we turn our attention from the lattice kissing number problem to the translative kissing number problem, which is known to be a more challenging problem. The idea behind Theorem \ref{thm:main0} is that a maximum sized ball inscribed in a cross-polytope need only be enlargened by a factor of approximately $r=\frac{2e}{\sqrt{2\pi e}}\approx 1.315$ (for large $n$) in order for it to have a volume equal to that of the cross-polytope; or more precisely, $\text{vol}(K_n)\sim \frac{1}{\sqrt{2}} \text{vol}(\frac{r}{\sqrt{n}} B_n)$ as $n\to\infty$. Recall that $\text{vol}(K_n)=\frac{2^n}{n!}$.

Thus, if the inscribed balls are enlargened by a factor of $\sqrt{2}$, and if most of the balls are contained in a cross-polytope of relatively small radius, then a volume comparison should give a better bound than $3^n-1$. Here we will use the fact that it is relatively simple to bound the volume of a sphere outside of a cross-polytope by considering certain spherical caps. Since the enlargened balls need not be disjoint, we attach a density to each ball such that the density at any given point of space does not exceed one, as in Blichfeldt's method.

In our estimations, the following limit will appear often: for a fixed real number $0\leq a\leq 1$, $\lim_{n\to\infty}\frac{1}{n}\log_2 \binom{n}{\lfloor an\rfloor}=H(a)$, where $H(s)=-s\log_2(s)-(1-s)\log_2(1-s)$ and $H(0)=0=H(1)$ is the binary entropy function defined for $0\leq s\leq 1$.

\begin{proof}[Proof of Theorem \ref{thm:main0}]
    Let $X=\{x_1,\dots,x_N\}$ be a cross-polytope kissing configuration, so that the interiors of $x_i+\frac{1}{2}K_n$ and $x_j+\frac{1}{2}K_n$ are disjoint for all $i\neq j$. For $x\in X$, define $I_x=\{i: |x_i|\geq \frac{b}{n}\}\subseteq \{1,\dots,n\}$ and let $X'=\{x\in X: |I_x|\geq \lfloor cn\rfloor \}$. We select parameter values $b=0.334$ and $c=0.296$. For each $x\in X\setminus X'$, assign a set $S_x\subseteq \{1,\dots,n\}\setminus I_x$ such that $|S_x|=n-\lfloor cn\rfloor$. Then $\sum_{i\in S_x} |x_i|\leq b(1-c)+\frac{b}{n}=a$. Thus, if $x,x'\in X\setminus X'$ are distinct points such that the sets $S_x$, $S_{x'}$ coincide, then $\sum_{i\not\in S_x} |x_i-x_i'|\geq 1-2a$, since otherwise $\sum_{i=1}^n |x_i-x_i'|\leq \sum_{i\not \in S_x} |x_i-x_i'|+\sum_{i\in S_x}(|x_i|+|x_i'|)<1-2a+2a=1$, a contradiction. Hence, the interiors of the $\lfloor cn\rfloor$-dimensional bodies $(x_i)_{i\not \in S_x}+\frac{1-2a}{2}K_{\lfloor cn \rfloor}$ and $(x_i')_{i\not\in S_{x'}}+\frac{1-2a}{2}K_{\lfloor cn \rfloor}$ are disjoint, and both are contained in $\frac{3-2a}{2}K_{\lfloor cn\rfloor}$. By a volume comparison, we obtain 
    \begin{align*}
        |X\setminus X'|&\leq \binom{n}{\lfloor cn \rfloor}\left(\frac{3-2a}{1-2a}\right)^{\lfloor cn \rfloor}\leq2^{H(c)(1+o(1))n}\left(\frac{3-2a}{1-2a}\right)^{\lfloor cn \rfloor}\\
        &\leq 2.9161^{(1+o(1))n}.
    \end{align*} 
    Next we will bound $M=|X'|$ using Blichfeldt's enlargement method. Note that $\frac{1}{\sqrt{n}}B_n\subset K_n$, so the interiors of $x_i+\frac{1}{2\sqrt{n}}B_n$ and $x_j+\frac{1}{2\sqrt
    n}B_n$ are disjoint for all $i\neq j$. Let $\delta_n(r)=1-2n r^2$ if $0\leq r\leq \frac{1}{\sqrt{2n}}$ and $\delta_n(r)=0$ else. The inequality of Blichfeldt \cite{Blichfeldt1929} (see also the proof of \cite[Theorem 29.2]{gruber2007convex}) yields that for every $y\in\mathbb{R}^n$ and every subset $\mathcal{X}\subseteq X'$,
    \begin{align*}
        2|\mathcal{X}|\sum_{x\in\mathcal{X}}\norm{y-x}_2^2 \geq \sum_{x,x'\in\mathcal{X}}\norm{x-x'}_2^2\geq \frac{|\mathcal{X}|(|\mathcal{X}|-1)}{n},
    \end{align*}
    which implies that 
    \begin{align*}
    \sum_{x\in X'} \delta_n(\norm{y-x}_2)=\sum_{x\in \mathcal{X}_y} \delta_n(\norm{y-x}_2)&\leq 1 \qquad\text{for all $y\in\mathbb{R}^n$,}
    \end{align*}
where $\mathcal{X}_y=\{x\in X':\norm{y-x}_2\leq \frac{1}{\sqrt{2n}}\}$. Let $R=1.5675$ and consider the following bound for the volume of $RK_n$:
\begin{align*}
    \text{vol}(RK_n) &\geq \int_{RK_n} \sum_{x\in X'}\delta_n(\norm{y-x}_2)dy\\
    &= \sum_{x\in X'}\int_{RK_n}\delta_n(\norm{y-x}_2)dy\\
    &=\sum_{x\in X'}\int_{(RK_n-x)\cap (\frac{1}{\sqrt{2n}}B_n)}\delta_n(\norm{y}_2)dy\\
    &\geq M\int_{\frac{1}{\sqrt{2n}}B_n} \delta_n(\norm{y}_2)dy-\sum_{x\in X'}\text{vol}(\frac{1}{\sqrt{2n}}B_n\setminus (RK_n-x))\\
    &=\frac{2M}{2+n}\text{vol}(\frac{1}{\sqrt{2n}}B_n)-\sum_{x\in X'}\text{vol}((x+\frac{1}{\sqrt{2n}}B_n)\setminus RK_n).
\end{align*}
Let $\{v_i\}_{i=1}^{2^n}$ be an enumeration of the vectors $(\pm 1,\dots,\pm 1)$, and define the half-spaces $\mathcal{H}_i=\{y\in\mathbb{R}^n: y\cdot v_i\geq R\}$, which bound the body $RK_n$. Let $x\in X'$ be fixed, and note that
\begin{align*}
\text{vol}((x+\frac{1}{\sqrt{2n}}B_n)\setminus RK_n) \leq \sum_{i=1}^{2^n}\text{vol}(\mathcal{H}_i \cap(x+\frac{1}{\sqrt
2n} B_n)).
\end{align*}
Note that $\mathcal{H}_i\cap(x+\frac{1}{\sqrt
2n} B_n)$, if non-empty, is a spherical cap of radius $r=\frac{1}{\sqrt{2n}}$ and height $h=r-r_0$, where $r_0=\frac{R-x\cdot v_i}{\sqrt{n}}$ is the distance from $x$ to the hyperplane bounding $\mathcal{H}_i$. Recall that $x$ has at least $\lfloor cn \rfloor$ coordinates, say coordinates $1,\dots, \lfloor cn\rfloor$, of absolute value $\geq \frac{b}{n}$. Thus, $x\cdot v_i\leq 1-\frac{2kb}{n}$, where $k$ is the number of $j=1,\dots,\lfloor cn\rfloor$ such that $x_j\cdot v_{ij}<0$. Consequently, $r_0\geq \frac{R-1+\frac{2kb}{n}}{\sqrt{n}}$. Approximating a spherical cap by an appropriate cylinder it is easy to see that $V_{\text{cap}}(r,h)\leq h(r^2-(r-h)^2)^{\frac{n-1}{2}} \text{vol}(B_{n-1})=V_{\text{cyl}}(r,h)$. Recall that $\frac{\text{vol}(B_n)}{\text{vol}(B_{n-1})}\sim \sqrt{\frac{2\pi}{n}}$. We obtain the upper bound
\begin{align*}
  \sum_{i=1}^{2^n}\text{vol}(\mathcal{H}_i \cap(x+\frac{1}{\sqrt
2n} B_n))&\leq \sum_{k=0}^{\lfloor \frac{(1+\frac{1}{\sqrt{2}}-R)n}{2b}\rfloor}\binom{\lfloor cn\rfloor}{k}2^{n-\lfloor cn\rfloor}V_{\text{cyl}}(r,r-\frac{R-1+\frac{2kb}{n}}{\sqrt{n}})\\
&\leq \alpha^{(1+o(1))n}\text{vol}(\frac{1}{\sqrt{2n}}B_n),
\end{align*}
where  $\alpha=\sup_{0\leq k' \leq \frac{1+\frac{1}{\sqrt{2}}-R}{2b}}2^{cH(\frac{k'}{c})+(1-c)}(1-2(R-1+2k'b)^2)^{\frac{1}{2}}<1$.

In conclusion,
\begin{align*}
    R^n \frac{2^n}{n!}=&\text{vol}(R K_n)\geq M\left(\frac{2}{2+n}\text{vol}(\frac{1}{\sqrt{2n}}B_n)-\alpha^{(1+o(1))n}\text{vol}(\frac{1}{\sqrt{2n}}B_n)\right)
\end{align*}
from which it follows that $M\leq 2.91616^{(1+o(1))n}$, by routine estimation.
\end{proof}

\subsection{An asymptotic lower bound for $\kappa_T(K_n)$}

In the other direction, Larman and Zong showed that $\kappa_T(K_n)\geq 1.1249^{(1-o(1))n}$ \cite{larman1999kissing}, which was consequently improved to $\kappa_T(K_n)\geq \kappa_T(K_n)\geq 1.1348^{(1-o(1))n}$ by Talata \cite{talata2000}. The latter result is also contained in \cite{xu2007note}, where the author also provides weaker constructive bounds based on algebraic geometry codes. The main result of this section is stated below.

\begin{thm}
    \label{thm:main}
    $\kappa_T(K_n)\geq1.1637^{(1-o(1))n}$ as $n\to\infty$.
\end{thm}

Let $m>0$ and let $\mathcal{S}$ be a finite subset of the boundary $\partial(m K_n)$ such that for all $x\in\mathcal{S}$, the cardinality of the set $\mathcal{B}(x)=\{y\in \mathcal{S}:\norm{x-y}_1<m\}$ is independent of the center $x$. This is achieved when $\mathcal{S}$ is the set of signed permutations of one vector $v\in\mathbb{R}^n$. Then by the union bound, \begin{align}
    \kappa_T(K_n)\geq \frac{|\mathcal{S}|}{|\mathcal{B}(x)|}.
\end{align}
Thus, the relevant problem is to choose the set $\mathcal{S}$ such that the probability that two uniformly random $x,x'\in\mathcal{S}$ satisfy $\norm{x-x'}_1<m$, is as small as possible. The results mentioned in the previous paragraph are obtained using the set $\mathcal{S}=\{x\in \{0,\pm 1\}^n:\sum_{i=1}^n |x_i|=m\}$ for a suitable integer $m$. However, since we are working with the $\ell^1$-norm, it is possible to decrease the aforementioned probability by replacing the set $\{0,\pm1\}$ with a larger set. Here we will consider for simplicity the set $\{0,\pm1,\pm 2\}$, which already gives a noticeable improvement.

\begin{proof}

Let $n>0$ and $m_1\geq m_2>0$ be integers to be determined later and define the set
\begin{align*}
    \mathcal{X}_{m_1,m_2}&=\{x\in\{0,\pm1,\pm2\}^n: |x_i|=j\text{ for $m_j$ values of $i$  ($j=1,2$)}\}
\end{align*}
of cardinality $|\mathcal{X}_{m_1,m_2}|=\binom{n}{m_1}\binom{n-m_1}{m_2}2^{m_1+m_2}$. Note that $\norm{x}_1=m_1+2m_2:=m$ for all $x\in \mathcal{X}_{m_1,m_2}$. For $x\in \mathcal{X}_{m_1,m_2}$, let $\mathcal{B}_{m_1,m_2}(x)=\{y\in\mathcal{X}_{m_1,m_2}:\norm{x-y}_1<m\}$. Clearly, $|\mathcal{B}_{m_1,m_2}(x)|$ does not depend on the center $x$ so we may consider the ball centered at $x_0=(1,\dots,1,2,\dots,2,0,\dots,0)\in\mathcal{X}_{m_1,m_2}$. For $x\in\mathcal{X}_{m_1,m_2}$, let
\begin{align*}
    c_1(x)&=|\{1\leq i \leq m_1: x_i=2\}|,\\
    c_2(x)&=|\{m_1+1\leq i \leq m_1+m_2: x_i=2\}|,\\
    c_3(x)&=|\{1\leq i \leq m_1+m_2: x_i=1\}|.
\end{align*}
Then $x\in\mathcal{B}_{m_1,m_2}(x_0)$ if and only if $2c_1(x)+4c_2(x)+2c_3(x)>m$. We obtain the upper bound
\begin{align*}
    |\mathcal{B}_{m_1,m_2}(x_0)|&\leq \sum_{0\leq x_1+x_2\leq m_2}\binom{m_1}{x_1}\binom{m_2}{x_2}\binom{m_1+m_2-x_1-x_2}{h(x_1,x_2)}\\
    &\binom{n-x_1-x_2-h(x_1,x_2)}{m_1-h(x_1,x_2)}\binom{n-m_1-x_1-x_2}{m_2-x_1-x_2}2^{m_1+m_2-x_1-x_2-h(x_1,x_2)}\\
    &:= \sum_{0\leq x_1+x_2\leq m_2} g(x_1,x_2;m_1,m_2),
\end{align*}
where $h(x_1,x_2)=\lfloor\frac{m}{2}-x_1-2x_2\rfloor+1$ is the smallest integer satisfying $2x_1+4x_2+2h(x_1,x_2)>m$. We will consider $m_1,m_2$ satisfying $m_2<\frac{m_1}{2}$, whence $0\leq h(x_1,x_2)\leq m_1$ holds for all admissible $x_1,x_2$.

Now let $m_i=\lfloor n z_i\rfloor$ and $x_i=\lfloor n y_i\rfloor$ for $i=1,2$, where $0< z_1,z_2\leq 1$ are fixed parameters satisfying $z_2<\frac{z_1}{2}$ and $z_1+z_2<1$, and $y_1,y_2\geq 0$ are variables satisfying $0\leq y_1+y_2\leq z_2$. Define
\begin{align*}
f(y_1,y_2;z_1,z_2)=\lim_{n\to\infty} \frac{1}{n}\log_2(g(\lfloor ny_1\rfloor,\lfloor ny_2\rfloor;\lfloor nz_1\rfloor,\lfloor nz_2\rfloor)).
\end{align*}
We can now write down an explicit expression for $f$.
\begin{align*}
    &f(y_1,y_2;z_1,z_2) \\
    &= z_1 H\left(\frac{y_1}{z_1}\right)+z_2H\left(\frac{y_2}{z_2}\right)+(z_1+z_2-y_1-y_2)H\left(\frac{\frac{z_1}{2}+z_2-y_1-2y_2}{z_1+z_2-y_1-y_2}\right)\\
    &+(1+y_2-\frac{z_1}{2}-z_2)H\left(\frac{y_1+2y_2+\frac{z_1}{2}-z_2}{1+y_2-\frac{z_1}{2}-z_2}\right)\\
    &+(1-z_1-y_1-y_2)H\left(\frac{z_2-y_1-y_2}{1-z_1-y_1-y_2}\right)+\frac{z_1}{2}+y_2.
\end{align*}
We take $z_1^*=\frac{19}{100}$ and $z_2^*=\frac{9}{100}$, and find that $\sup_{0\leq y_1+y_2\leq z_2^*} f(y_1,y_2;z_1^*,z_2^*)\approx 1.17029$, attained at $(y_1,y_2)\approx(0.01728,0.04327)$. We conclude that for $m_i^*=\lfloor n z_i^*\rfloor$, 
\begin{align*}
    \liminf_{n\to\infty}&\frac{1}{n}\log_2\left(\kappa_T(K_n)\right) \geq \lim_{n\to\infty} \frac{1}{n}\log_2 |\mathcal{X}_{m_1^*,m_2^*}|-\limsup_{n\to\infty} \frac{1}{n}\log_2 |\mathcal{B}_{m_1^*,m_2^*}(x_0)|\\
    &\geq H(z_1^*)+(1-z_1^*) H\left(\frac{z_2^*}{1-z_1^*}\right)+z_1^*+z_2^*-\sup_{0\leq y_1,y_2\leq z_2^*} f(y_1,y_2;z_1^*,z_2^*)\\
    &\approx 0.218818,
\end{align*}
which shows that $\kappa_T(K_n)\geq 2^{0.21881(1-o(1))n}\geq 1.1637^{(1-o(1))n}$.
\end{proof}

It is interesting to note that Theorem \ref{thm:main} gives a better existence result for the kissing number of the cross-polytope than the best known existence result $\kappa_T(B_n)\geq 2^{0.2075(1+o(1))n}$ for the $\ell^2$-ball mentioned in the introduction \cite{wyner1965capabilities}.


\begin{thebibliography}{10}

\bibitem{schutte1952problem}
K.~Sch{\"u}tte and B.~L. van~der Waerden, ``Das {P}roblem der dreizehn {K}ugeln,'' {\em Mathematische Annalen}, vol.~125, no.~1, pp.~325--334, 1952.

\bibitem{musin2008kissing}
O.~R. Musin, ``The kissing number in four dimensions,'' {\em Annals of Mathematics}, pp.~1--32, 2008.

\bibitem{levenshtein1979bounds}
V.~I. Levenshtein, ``On bounds for packings in $n$-dimensional {E}uclidean space,'' in {\em Doklady Akademii Nauk}, vol.~245, pp.~1299--1303, Russian Academy of Sciences, 1979.

\bibitem{odlyzko1979new}
A.~M. Odlyzko and N.~J. Sloane, ``New bounds on the number of unit spheres that can touch a unit sphere in $n$ dimensions,'' {\em Journal of Combinatorial Theory, Series A}, vol.~26, no.~2, pp.~210--214, 1979.

\bibitem{wyner1965capabilities}
A.~D. Wyner, ``Capabilities of bounded discrepancy decoding,'' {\em Bell System Technical Journal}, vol.~44, no.~6, pp.~1061--1122, 1965.

\bibitem{kabatiansky1978bounds}
G.~A. Kabatiansky and V.~I. Levenshtein, ``On bounds for packings on a sphere and in space,'' {\em Problemy {P}eredachi {I}nformatsii}, vol.~14, no.~1, pp.~3--25, 1978.

\bibitem{Swanepoel2018}
K.~J. Swanepoel, {\em Combinatorial Distance Geometry in Normed Spaces}, pp.~407--458.
\newblock Berlin, Heidelberg: Springer Berlin Heidelberg, 2018.

\bibitem{larman1999kissing}
D.~Larman and C.~Zong, ``On the kissing numbers of some special convex bodies,'' {\em Discrete \& Computational Geometry}, vol.~21, pp.~233--242, 1999.

\bibitem{hadwiger1957treffanzahlen}
H.~Hadwiger, ``{\"U}ber {T}reffanzahlen bei translationsgleichen {E}ik{\"o}rpern,'' {\em Archiv der Mathematik}, vol.~8, no.~3, pp.~212--213, 1957.

\bibitem{Blichfeldt1929}
H.~Blichfeldt, ``The minimum value of quadratic forms, and the closest packing of spheres,'' {\em Mathematische Annalen}, vol.~101, pp.~605--607, 1929.

\bibitem{minkowski1961diophantische}
H.~Minkowski, {\em Diophantische Approximationen}.
\newblock Leipzig: Teubner, 1907.

\bibitem{gruber2007convex}
P.~M. Gruber, {\em Convex and discrete geometry}, vol.~336.
\newblock Springer, 2007.

\bibitem{miller2019kissing}
S.~D. Miller and N.~Stephens-Davidowitz, ``Kissing numbers and transference theorems from generalized tail bounds,'' {\em SIAM Journal on Discrete Mathematics}, vol.~33, no.~3, pp.~1313--1325, 2019.

\bibitem{talata2000}
I.~Talata, ``A lower bound for the translative kissing numbers of simplices,'' {\em Combinatorica}, vol.~20, no.~2, pp.~281--293, 2000.

\bibitem{xu2007note}
L.~Xu, ``A note on the kissing numbers of superballs,'' {\em Discrete \& Computational Geometry}, vol.~37, pp.~485--491, 2007.

\end{thebibliography}
\end{document}